\documentclass[12pt]{amsart}
\usepackage{amsfonts}
\usepackage{amsmath}
\usepackage{amssymb,latexsym}
\usepackage[mathcal]{eucal}
\usepackage{amscd}

\usepackage{makecell}

\usepackage[pdftex,bookmarks,colorlinks,breaklinks]{hyperref}
\input xy
\xyoption{all}

\newcommand{\ch}{\mathbf{1}}

\newcommand{\Z}{\mathbb{Z}}

\newcommand{\R}{\mathbb{R}}
\newcommand{\C}{\mathbb{C}}

\newcommand{\ga}{\gamma}
\newcommand{\del}{\delta}
\newcommand{\Del}{\Delta}
\newcommand{\ep}{\epsilon}
\newcommand{\sig}{\sigma}

\newcommand{\ol}{\overline}

\newcommand{\rest}{\upharpoonright}

\newcommand{\cls}{{\rm{cls\,}}}
\newcommand{\co}{{\rm{co\,}}}

\newcommand{\ext}{{\rm{ext\,}}}

\newcommand{\id}{{\rm{id}}}

\newcommand{\spann}{{\rm{span}}}

\newcommand{\ins}{{\rm{ins}}}
\newcommand{\out}{{\rm{out}}}

\newcommand{\Homeo}{\rm{Homeo}}

\swapnumbers
\theoremstyle{plain}
\newtheorem{thm}{Theorem}[section]
\newtheorem{cor}[thm]{Corollary}
\newtheorem{lem}[thm]{Lemma}
\newtheorem{prop}[thm]{Proposition}

\theoremstyle{definition}
\newtheorem{defn}[thm]{Definition}
\newtheorem{defns}[thm]{Definitions}

\newtheorem{rmk}[thm]{Remark}

\newtheorem{rem}[thm]{Remark}
\newtheorem{rems}[thm]{Remarks}

\newtheorem{exa}[thm]{Example}
\newtheorem{exas}[thm]{Examples}
\newtheorem{prob}[thm]{Problem}

\begin{document}

\title[Affinely prime dynamical systems]
{Affinely prime dynamical systems}

\begin{abstract}

\end{abstract}

\author{Hillel Furstenberg, Eli Glasner and Benjamin Weiss}

\address {Institute of Mathematics\\
Hebrew University of Jerusalem\\
Jerusalem\\
Israel}
\email{harry@math.huji.ac.il}

\address{Department of Mathematics\\
     Tel Aviv University\\
         Tel Aviv\\
         Israel}
\email{glasner@math.tau.ac.il}

\address {Institute of Mathematics\\
Hebrew University of Jerusalem\\
Jerusalem\\
Israel}
\email{weiss@math.huji.ac.il}




\date{November 4, 2015}

\begin{abstract}
We study representations of groups by ``affine" automorphisms of compact, convex spaces, with special focus on ``irreducible" representations: equivalently ``minimal" actions. When the group in question is 
$PSL(2,\R)$, we exhibit a one-one correspondence between bounded harmonic functions on the upper half-plane and a certain class of irreducible representations.  Our analysis shows that, surprisingly, all these representations are equivalent.  In fact we find that all irreducible affine representations of this group are equivalent. The key to this is a property we call ``linear Stone-Weierstrass" for group actions on compact spaces, which, if it holds for the ``universal strongly proximal space" of the group (to be defined)
then the induced action on the space of probability measures
on this space is the unique irreducible affine representation of the group.
\end{abstract}

\maketitle

\section*{}

The classical theory of group representations deals with representing a group as automorphisms of vector spaces. In principle one can take any category with its morphisms and study representing  a group by automorphisms of objects in this category.  In what follows we shall do this for the category of compact convex spaces with morphisms preserving the affine structure.  Here too there is particular interest in the ``irreducible" representations where no proper ``subobject" is invariant under the action. A pleasant aspect of this theory is that for any group there is a ``universal" irreducible representation from which all others can be derived. Moreover, for many groups, the universal irreducible representation can be described explicitly.  Following our preliminary discussion 
we focus on the group $PSL(2,\R)$, or equivalently, on the M\"{o}bius group of analytic maps preserving the unit disc of the complex plane.
Denote the latter group by $G$.  We show, following \cite{F1}, that each bounded harmonic function on the disc leads to an irreducible representation of $G$ on a compact, convex subset of $L^\infty(G)$.  Since there is an abundance of bounded harmonic functions on the disc we might expect to find a great variety of non-equivalent irreducible representations of $PSL(2,\R)$.
This was our initial guess and the motivation for the ensuing research.
As it turns out, the universal irreducible representation of the M\"{o}bius group is given by the natural action on probability measures on the unit circle.  Moreover we show that this representation is ``prime", meaning that no other irreducible representation can be derived from this one.  
This means, in particular, that all non constant harmonic functions lead to equivalent irreducible representations.

In the first section we develop the rudiments of the theory of irreducible affine dynamical systems
and introduce the notion of an affinely prime dynamical system.
In the second section we consider the group $G$ of M\"{o}bius transformations preserving the unit disk 
$D \subset \C$, which is topologically isomorphic to the group $PSL(2,\R)$.
As was shown in \cite{F1},
the action of $G$ on the boundary $S^1$ of $D$ is minimal and strongly proximal
and moreover the system $(S^1,G)$ is the universal minimal and strongly proximal $G$-action,
denoted as $\Pi_s(G)$.
This is the same as saying that the induced action of $G$ on the space $M(S^1)$ of
probability measures on $S^1$ is the universal irreducible affine action of $G$. 
We prove that in fact, up to affine isomorphism, the irreducible affine system 
$(M(\Pi_s(G)),G) = (M(S^1),G)$ is the unique irreducible affine $G$-system. 
In the last section we show, following \cite{F1},
that there is a one-to-one correspondence between bounded harmonic functions $h$ on 
the unit disk $D \cong G/K$ (where $K \subset G$ is the subgroup of rotations of $D$) 
and irreducible affine systems $(Q_h,G)$ in $L^\infty(G)$,
where each such irreducible system contains a unique $K$ invariant function
which is the lift of $h$ from $G/K$ to $G$. Moreover, as a consequence of the analysis of the previous 
section, all the affine systems $Q_h$ are isomorphic to the universal
irreducible affine system $(M(S^1),G) = (M(\Pi_s(G)),G)$.

We thank David Kazhdan and Erez Lapid for several helpful conversations that, eventually, led us to a simpler 
and more elegant proof of Theorem \ref{main}.
\vspace{.5cm}

\section{Affinely prime dynamical systems}

A dynamical system $(X,G,\psi)$ is a triple consisting of a compact metric space $X$,
a topological group $G$ and a continuous homomorphism $\psi : G \to \Homeo(X)$, the Polish 
group of homeomorphisms of $X$ equipped with the compact open topology.
As a rule we will suppress the homomorphism $\psi$ and, given $x \in X$ and $g \in G$,
write $gx$ for $\psi(g)(x)$. 
A dynamical system is {\em nontrivial} when it contains more than one point.
Given two $G$ dynamical systems a {\em homomorphism} 
$\pi : (X,G) \to (Y,G)$ is a continuous map of $X$ into $Y$ which intertwines the $G$-actions.
When $\pi$ is surjective we say that it is a {\em factor map} and that $Y$ is a {\em factor} of $X$.
A dynamical system $(X,G)$ is {\em prime} if every factor map 
$\pi : (X,G) \to (Y,G)$ with $Y$ nontrivial is one-to-one.

If $(X,G)$ is a dynamical system and $Y \subset X$ is an invariant closed subset we say that
$(Y,G)$, the restriction of the action of $G$ to $Y$, is a {\em subsystem}.
When $(X,G)$ has no proper subsystems we say that it is {\em minimal}.
This is of course the case if and only if the orbit $Gx$ of every point $x \in X$ is dense.
We say that two points $x, y$ in a system $X$ are {\em proximal} if there
exists a point $z \in X$ and a sequence $g_n \in G$ such that $\lim g_n x = \lim g_n y = z$.
The system $(X,G)$ is {\em proximal} if every pair of points in $X$ is proximal.

\begin{lem}\label{prime}
A nontrivial prime dynamical system is either minimal or it has a 
unique fixed point and every other point has a dense orbit.
\end{lem}

\begin{proof}
Let $(X,G)$ be a nontrivial prime dynamical system.
If $X$ properly contains a closed $G$-invariant subset $Y \subsetneq X$, 
which contains more than one point,
form the set
$$
R = (Y \times Y) \cup \{(x,x) : x \in X\} \subsetneq X \times X.
$$
This is an icer (i.e. an invariant closed equivalence relation) on 
$X$, and the corresponding homomorphism $ \pi : X \to X/R$ is non-trivial, 
contradicting primality.

Thus every proper closed invariant subset of $X$ is a singleton.
It follows that if $X$ is not minimal then it has a unique fixed point
and every other point has a dense orbit, as claimed.
\end{proof}

The space $M(X)$ of probability measures on $X$ will be equipped with its natural
weak$^*$ topology which is inherited from $C(X)^*$ when a measure is identified
with the corresponding linear functional on $C(X)$, the Banach algebra of 
real valued continuous functions on $X$. The compact metric space $M(X)$ also supports
an affine structure and the $G$-action on $X$ induces a continuous
affine action of $G$ on $M(X)$. In general if $Q$ is a compact convex 
metrisable subset of a locally convex topological vector space and $G$
acts on $Q$ as a group of continuous affine maps
(i.e. each $g \in G$ preserves convex combinations), we say that $(Q,G)$ is an
{\em affine dynamical system}. 

For more details on the notions and results introduced below see e.g. \cite{Gl}.

\begin{defns}
\begin{enumerate}
\item
Let $(X,G)$ be a dynamical system and $(Q,G)$ an affine dynamical system.
We say that $Q$ is an {\em affine compactification} of $X$ if
there is a homomorphism $\phi : X \to Q$ such that $\ol{co}(\phi(X)) = Q$,
where for $A \subset Q$, $\ol{co}(A)$ denotes the closed convex hull of the set $A$.
When $\phi$ is one-to-one we say that it is {\em faithful}  (or that it is an
{\em affine embedding}).
\item
An affine dynamical system $(Q,G)$ is {\em irreducible} if it does not contain properly any
affine subsystem; i.e. if whenever $Q' \subset Q$ is a closed convex and $G$-invariant subset
then $Q' =Q$.
\item
An affine dynamical system $(Q,G)$ is {\em affinely prime} if it does not admit any proper
factor affine system; i.e. if whenever $\pi : Q \to Q'$ is an affine surjective
homomorphism with $Q'$ nontrivial, then $\pi$ is one-to-one.
\item
A dynamical system $(X,G)$ is {\em affinely prime} if with respect to the canonical 
faithful affine compactification $\phi : X \to M(X)$ given by $\phi(x) = \delta_x$,
the associated affine system $(M(X),G)$ is affinely prime.
\item
A dynamical system $(X,G)$ is {\em strongly proximal} if for every $\mu \in M(X)$
there is a sequence of elements $g_n \in G$ and a point $x \in X$ such that
$\lim g_n\mu = \del_x$, the point mass at $x$. 
\end{enumerate}
\end{defns}

\vspace{.5cm}

The next proposition follows easily from Choquet's theory; see e.g. \cite{P}.

\begin{prop}\label{aff}
\begin{enumerate}
\item
If $Q$ is an affine dynamical system and $X = \ol{\ext (Q)}$ 
(where $\ext(Q)$ denotes the set of extreme points of $Q$), then 
$Q$ is a faithful affine compactification of $X$.
\item
For a dynamical system $(X,G)$ the canonical affine compactification defined on
$(M(X),G)$ is {\em universal}; i.e. for any affine compactification $\phi : X \to Q$ there is a 
uniquely defined (barycenter) map $\beta : M(X) \to Q$ with $\beta (\del_x) = \phi(x)$
for every $x \in X$.
\end{enumerate}
\end{prop}

\begin{lem}\label{Mil}
If $(Q,G)$ is an irreducible
affine system and $A \subset Q$ is any closed $G$-invariant subset, then $A$
contains $\ext(Q)$.
\end{lem}

\begin{proof}
The barycenter map takes $M(A)$ onto $Q$; so, in particular, each
extremal is the barycenter of a measure on $A$ which by extremality
must be the corresponding point mass.
\end{proof}

\vspace{.5cm}

\begin{lem}\label{sp}
For a dynamical system $(X,G)$ the affine compactification $M(X)$ is irreducible if and only if $(X,G)$ is
minimal and strongly proximal.
\end{lem}

\begin{proof}
If $Y \subset X$ is a proper closed invariant subset
then $M(Y) \subsetneq M(X)$. Thus irreducibility of $M(X)$ implies minimality of $X$.
Given any element $\mu \in M(X)$ let $Z_\mu = \ol{\{g \mu : g \in G\}}$
and $Q_\mu = \ol{co}(Z_\mu)$. The latter is an affine sub-system of $M(X)$. 
If $M(X)$ is irreducible it follows that $Q_\mu = M(X)$.
From 
Lemma \ref{Mil}
we have $Z_\mu \supset \ext(M(X)) = \{\del_x : x \in X\}$,
whence $X$ is strongly proximal.

Conversely if $(X,G)$ is minimal and strongly proximal then it is easy
to see that every $Q_\mu = M(X)$; i.e. $M(X)$ is irreducible.
\end{proof}

\vspace{.5cm}

In the following lemma we recall some basic facts about affine systems
and also provide the short proofs.

\begin{lem}
\begin{enumerate}
\item
A proximal system contains exactly one minimal subsystem.
\item
A minimal proximal system admits 
no endomorphisms other than the identity automorphism.
\item
A system $(X,T)$ is strongly proximal if and only if the system $(M(X),G)$ is proximal.
In particular, a strongly proximal system is proximal.
\item
For an affine irreducible system $(Q,G)$ let $X$ denote the
closure of the extreme points of $Q$. Then $X$ is the unique minimal subsystem of $Q$
and the system $(X,G)$ is strongly proximal.
\item
If there is a homomorphism $\pi : Q \to P$, where $(Q,G)$ and $(P,G)$ 
are irreducible affine systems then it is unique.
In particular, the only affine endomorphism of an irreducible affine system
is the identity.
\end{enumerate}
\end{lem}

\begin{proof}
(1)
By Zorn's lemma every dynamical system contains at least one minimal subsystem.
But if $x, y \in X$ belong to two distinct minimal subsystems they can not be proximal.

(2)
Suppose $(X,G)$ is minimal and proximal and that $\phi : X \to X$ is an endomorphism.
Since the pair $(x, \phi(x))$ is proximal, there is a sequence $g_n \in G$ with 
$\lim g_n(x, \phi(x)) = (z,z)$ for some $z \in X$, whence $z = \phi(z)$.
Since $X$ is minimal this implies that $\phi = \id$.

(3) Clearly proximality of $M(X)$ implies strong proximality of $X$.
Conversely, let $(X,G)$ be a strongly proximal system. Given $x, y \in X$
form the measure $\mu = \frac12(\del_x + \del_y)$. There exists a point $z \in X$
and a sequence $g_n \in G$ with $\lim \mu = \del_z$ and, as $\del_z$
is an extreme point of $M(X)$, it is easy to see that this implies that 
$\lim g_n x = \lim g_n y =z$. Thus any two points in $X$ are proximal;
i.e. $X$ is a proximal system. It is now easy to see that $M(X)$ is also proximal.

(4)
By Proposition \ref{aff} there is an affine surjective homomorphism $\beta : M(X) \to Q$.
Given $\mu \in M(X)$ let $Z_\mu = \ol{\{g \mu : g \in G\}}$ and $Q_\mu = \ol{co}(Z_\mu)$.
Then, by the irreducibility of $Q$ we have $\beta(Q_\mu) = Q$. In particular, for
every extreme point $w \in \ext (Q)\subset X$ there is $\nu \in Q_\mu$ with $\beta(\nu) = w$.
As $w$ is an extreme point this implies that $\nu =\del_w \in X$. 
It follows that $X \subset Q_\mu$, whence $Q_\mu = M(X)$.
Thus $M(X)$ is also irreducible and an application of Lemma \ref{sp} concludes the proof.   

(5)
Suppose $\pi : Q \to P$ and $\sig : Q \to P$ are two affine homomorphisms.
Let $X = \cls (\ext (Q))$ and $Y = \cls(\ext(P))$. We know that both $X$ and $Y$
are proximal and minimal systems.
For every $x \in X$ we consider the pair $(\pi(x), \sig(x))$. This is a proximal pair in $Y$ 
and thus for some sequence $g_n \in G$ we have 
$\lim (g_n \pi(x), g_n\sig(x)) = (y,y)$ for some $y \in Y$. However we can also
assume that the limit $\lim g_n x = z \in X$ exists and then
$(y,y) = (\pi(z), \sig(z))$, hence $\pi(z) = \sig(z)$. $X$ being minimal, this implies
that $\pi(z') = \sig(z')$ for every $z' \in X$ and finally, as $\pi$ and $\sig$ are affine maps, 
this leads to the conclusion that $\pi = \sig$.
\end{proof}

\vspace{.5cm}

For any topological group $G$ there exists a {\em universal minimal strongly proximal
system} which we denote by $\Pi_s(G)$. 
Recalling the fact that a group $G$ is amenable if and only every compact
dynamical system $(X,G)$ admits an invariant probability measure,
we see that a group $G$ is amenable if and only if the space $\Pi_s(G)$ is a trivial one point space.
The following is a consequence of (4):

\begin{cor}\label{unique}
The affine dynamical system $(M(\Pi_s(G)), G)$ is irreducible and it 
is the universal affine system for irreducible affine $G$ systems.
I.e. for any irreducible affine $G$ system $Q$
there is a unique surjective affine homomorphism $\Theta : M(\Pi_s(G)) \to Q$. 
In particular if $\Pi_s(G)$ is affinely prime then $M(\Pi_s(G))$ is the only 
nontrivial irreducible affine $G$-system.
\end{cor}

\vspace{.5cm}

The next definition is reminiscent of the classical Stone-Weierstrass theorem.

\begin{defn}\label{Vf}
We say that a dynamical system $(X,G)$ has the {\em linear Stone-Weierstrass property
(LSW)} if for every non-constant function $f \in C(X)$ the uniformly closed linear span
$V_f $ of the set $\{f^g : g \in G\}\cup \{\ch\}$ is all of $C(X)$
(here $f^g(x) = f(gx)$).
\end{defn}

\begin{prop}
A dynamical system has LSW if and only if it is affinely prime.
\end{prop}

\begin{proof}
For a function $f \in C(X)$ we denote by $\hat{f} \in Af\!f(M(X))$ the map
$$
\mu \mapsto \int f \, d\mu, \qquad \mu \in M(X).  
$$
Suppose first that $X$ has the LSW property and let $\pi : M(X) \to Q$ be
an affine homomorphism with nontrivial $Q$. 
Let $Af\!f(Q)$ denote the collection of continuous affine real valued
functions on $Q$ and let 
$$
\mathcal{A}(Q) = \{f \in C(X) : \hat{f} = F \circ \pi, \ {\text{for some}}\ F \in Af\!f(Q)\}.
$$
The LSW property implies that $\mathcal{A}(Q) = C(X)$.
Suppose now that $\pi(\mu) = \pi(\nu)$ and $\mu \not= \nu$. Then there is $f \in C(X)$
with $\hat{f}(\mu) \not= \hat{f}(\nu)$, and, as $\hat{f} = F \circ \pi, \ {\text{for some}}\ F \in Af\!f(Q)$,
we have
$\hat{f}(\mu)  = F \circ \pi(\mu) = F \circ \pi(\nu) =  \hat{f}(\nu)$, a contradiction.
Thus $\pi$ is indeed one-to-one.

Conversely, suppose $(X,G)$ is affinely prime and let $f$ be a non constant function in $C(X)$.
Let $V_f$ be as in Definition \ref{Vf}.
If $V_f$ is a proper subspace of $C(X)$ then the restriction map $\mu \to \mu \rest V_f,
\quad M(X) \to Q$, where the latter is  
$$
Q  = S(V_f) = \{\xi \in V_f^*, \  \xi \ge 0 \ {\text{on non negative functions, and }} \ \xi(\ch) =1\},
$$ 
the {\em state space} of $V_f$, yields a non-injective affine homomorphism of $M(X)$. 
\end{proof}

\vspace{.5cm}

\begin{prop}\label{aprime}
If $(X, G)$ is affinely prime then it is prime, whence
it is either minimal or it has a unique fixed point and every other point has a dense orbit.
\end{prop}

\begin{proof}
Observe that
if $\pi : (X, G) \to (Y,G)$ is a surjective but non-injective factor map then
the induced map $\pi_* : M(X) \to M(Y)$ is a surjective but non-injective
affine homomorphism. Thus an affinely prime system is prime.
The rest  follows by Lemma \ref{prime}.
\end{proof}

\vspace{.5cm}

\begin{defn}
We say that a dynamical system $(X,G)$ is {\em completely uniquely ergodic} if it
admits a unique $G$-invariant probability measure, say $\eta$, and 
$\{\eta\}$ is the only irreducible affine subsystem of $M(X)$.
\end{defn}

\begin{prop}\label{prop}
If $(X, G)$ is affinely prime then the dynamical system $(X,G)$ is:
\begin{enumerate}
\item
prime;
\item
it is either minimal or or it has a unique fixed point and every other point has a dense orbit;
\item
it is either completely uniquely ergodic or it is strongly proximal.
\item
For a minimal affinely prime system which is not completely uniquely ergodic,
$M(X)$ is irreducible. 
\end{enumerate}
\end{prop}

\begin{proof}
(1) Observe that
if $\pi : (X, G) \to (Y,G)$ is a surjective but non-injective factor map then
the induced map $\pi_* : M(X) \to M(Y)$ is a surjective but non-injective
affine homomorphism. Thus an affinely prime system is prime.

(2)
This now follows by Lemma \ref{prime}.

(3)
Assume that $X$ is not strongly proximal. Then there is a probability
measure $\xi \in M(X)$ whose orbit closure $Z_\xi =\cls (G\xi)$ does not meet
$X$. It follows that $Q_\xi = \ol{\co}(Z_\xi)\subsetneq M(X)$
is a nonempty closed convex and $G$-invariant proper subset of $M(X)$. 

Now given any nonempty closed convex and $G$-invariant proper subset $Q$ of $M(X)$,
set $L = w^*{\text{-}}\cls \spann (Q)\subset C(X)^*$.

Suppose first that $L = C(X)^*$. Then in particular every point mass $\del_x$ is in $L$
and there is a sequence $\phi_n \in \spann(Q - Q)$ such that $\phi_n \overset{w^*}{\to} \del_x$.
Let $\phi_n = a_n \mu_n - b_n \nu_n$ with $\mu_n, \nu_n \in Q$ and $a_n, b_n \ge 0$.
It follows that $b_n\nu_n \to 0$ and  $\mu_n \to \del_x$. We conclude that
$Q = M(X)$. Thus in this case $X$ is minimal and strongly proximal.

Suppose next that $L$ is a proper subspace of $C(X)^*$.
Fix some $\phi \in C(X)^* \setminus L$.
By the Hahn-Banach separation theorem (see e.g. Corollary 11 on page 418 of \cite{DS}) 
there is a function $f \in C(X)$ such that
$\phi(f) =1$ and $\psi(f) \ge 0$ for all $\psi \in L$. 
Since $L$ is a subspace, it follows that $\psi(f) = 0$ for all $\psi \in L$.

Thus $f$ is an element of the norm closed $G$-invariant
subspace $L_{\perp} \subset C(X)$ defined by:
$$
L_{\perp} = \{h \in C(X) : \psi(h) = 0, \ \forall \psi \in L\}.
$$
Next define $V = L_{\perp} \oplus \R \ch$, where the latter stands for the space
of constant functions.
If $V$ is a proper subspace of $C(X)$ this contradicts the
assumption that $X$ has the LSW property. So we now assume that $V = C(X)$.

{\bf Case 1:} There exists $Q$ as above which contains more than one element.

Let $\nu_1, \nu_2$ be two distinct elements of $Q$ and let $F \in C(X)$
be such that $\nu_1(F) \not= \nu_2(F)$. We write $F = h + c\ch$
with $h \in L_{\perp}$ and $c \in \R$ and then get:
\begin{gather*}
\nu_1(F) = \nu_1(h + c\ch) = \nu_1(h) + \nu_1(c\ch) = \nu_1(c\ch) = c,\ {\text{and}}\\
\nu_2(F) = \nu_2(h + c\ch) = \nu_2(h) + \nu_2(c\ch) = \nu_2(c\ch) = c,
\end{gather*}
a contradiction.

{\bf Case 2:} Every closed $G$-invariant convex proper subset of $M(X)$ is a singleton.

In that case the collection $K$ of $G$-invariant probability measures
is a closed convex $G$-invariant subspace of $M(X)$.
Now, as we assume that $(X,G)$ is not trivial, 
the case where $K$ is not a singleton can be ruled out, as in Case 1 above,
and we are left with the case where $K = Q = \{\eta\}$ is the only
closed convex $G$-invariant subset of $M(X)$, which is, by
definition, the case of complete unique ergodicity.

(4)
This follows from part (3) and Lemma \ref{sp}.
\end{proof}

The following diagram sums up the various possible situations described in Proposition \ref{prop}:
\begin{table}[h]
\begin{center}
\begin{tabular}{ | l | l | l | l | l | l | }
\end{tabular}
\caption{ \protect  Affinely prime systems}
\end{center}
\end{table}
{\small{
\begin{table}[h]
\begin{center}
\begin{tabular}{ |  p{3cm}  |  p{3cm}  |  p{4cm}  |  p{3cm}  | }
\hline 
\diaghead{\theadfont Diag ColumnmnHead II}
{$X$}{$M(X)$}
& irreducible &
reducible \\
\hline
minimal & minimal strongly \newline{proximal} & 
completely uniquely \newline{ergodic} \\ 
\hline
$x_0\ {\text{fixed,}}$ \newline{$ \ol{Gx} =X \ \forall x \not= x_0$}  & vacuous & 
$\{\delta_{x_0}\}$ the unique
\newline { irreducible subset of $M(X)$} \\
\hline
\end{tabular}
\end{center}
\end{table}
}}

\vspace{.5cm}

\begin{rem}
The converse of Proposition \ref{prop} (4), of course, does not hold. There are 
many minimal strongly proximal systems (so with $M(X)$ irreducible) which are
not even prime (see e.g. Examples \ref{exas2}  and \ref{exas3} below).
\end{rem}

\vspace{.5cm}

\begin{exas}\label{exas1}
\begin{enumerate}
\item
For every prime $p$ the map $Tx = x + 1 \pmod p$ generates
a prime system $(\Z_p,T)$. It is affinely prime (over $\R$) only for $p = 2, 3$.
\item
Let $X$ be the Cantor set and $G = \Homeo(X)$, the group of self-homeomorphisms of $X$.
The system $(X,G)$ has LSW.
\item
Let $X=S^2$, the two dimensional sphere in $\R^3$, and $G = \Homeo(X)$, 
the group of self-homeomorphisms of $X$. The system $(X,G)$ has LSW.
\item
Again take $X=S^2$, but now consider the action of $H < G= \Homeo(X)$,
the subgroup consisting of those homeomorphisms which fix the north pole.
Again $(X,H)$ is affinely prime, this time strongly proximal with a unique fixed point.
\item
Let $X = \Z \cup \{\infty\}$ be the one point compactification of the integers and $T$ the
translation $Tx = x +1$ on $\Z$ which fixes the point at infinity. It is easy to check that $X$ 
is prime and strongly proximal. However, it does not have the LSW property. 
\end{enumerate}

\begin{proof}

(1) Clear when one considers the associated Koopman representation on $C(\Z_p) 
\cong \R^p$.

(2) Let $f$ be a non constant function in $C(X)$. Rescaling we can assume that 
$0 \le f(x) \le 1$ for every $x \in X$ and that the values $0$ and $1$ are attained,
say $f(x_0)=0$ and $f(x_1)=1$.

Suppose 
$$
V_f  = \cls \spann \left(\{f^g : g \in G\}\cup \{\ch\}\right) \subsetneq C(X).
$$
Then there exists a functional $0 \not = \mu \in C(X)^*$
such that $\mu(h)=0$ for every $h \in V_f$. We think of $\mu$ as a signed measure
and write $\mu = \mu_0 - \mu_1$, where $\mu_0$ and $\mu_1$ are non-negative measures,
such that for some Borel set $B \subset X$, $\mu_0(B) = \mu_0(X)$ and
$\mu_1(X \setminus B) = \mu_1(X)$. Since $\ch \in V_f$ we have
$\mu(X) = \mu_0(X) - \mu_1(X) =0$, whence $\mu_0(X) = \mu_1(X) = a >0$.
Again with no loss of generality we assume that $\mu_0(X) = \mu_1(X) = a = 1$.

Given $0 < \ep < 1/8$ we can find 
closed subsets $K_0 \subset B$
and $K_1 \subset X \setminus B$ such that $\mu_i(K_i) > (1- \ep), \ i=0,1$. 

Next choose a sequence $g_n \in G$ such that $g_n(K_i) \to x_i,\ i=0,1$,
in the sense that for every two open neighbourhoods $U_i$ of $x_i$
there is $n_0$ with $g_n K_i \subset V_i$ for all $n \ge n_0$.

We also assume, as we may, that the limits $g_n\mu_i \to \nu_i,\ i=0,1$ exist,
and that $\nu_i = c_i \delta_{x_i} + (1 - c_i)\nu'_i$, where
$(1 - \ep)  < c_i  \le 1$ and the measures $\nu'_i$ are probability measures.

Now
\begin{gather*}
\int f \, dg_n\mu_0 \to c_0f(x_0) + (1 - c_0)\nu'_0(f) = (1 - c_0) \nu'_0(f) \le \ep \quad{\text{and}}\\
\int f \, dg_n\mu_1 \to c_1f(x_1) + (1 - c_1) \nu'_1(f) =  c_1 + (1 - c_1) \nu'_1(f) \ge 1- \ep.
\end{gather*}
But, as $f \in V_f$, these two limits are equal and we arrive at the absurd inequality
$7/8  < 1 - \ep \le \ep < 1/8$. 

(3)
As in the previous proof, given $f$ a non constant function is $C(X)$, we rescale $f$,
form the space $V_f$ and proceed as above. When we choose the closed disjoint sets $K_0, K_1$
we can assume that they are Cantor sets.
We claim that there is a smooth closed simple Jordan curve with
$A \subset \ins(\ga)$ and $B \subset \out(\ga)$.
In fact this follows easily e.g. from \cite[Proposition 1.8, page 4]{Co}.
Now we again proceed as in part (2) above and choose the homeomorphisms
$g_n$ so that their restriction to a sufficiently small neighborhood of $\ga$ is the identity.
The rest of the proof goes verbatim as in part (2).

(4) 
A similar argument.

(5)
In order to see this observe
first that $C(X) \cong c(\Z)$, the Banach space of converging sequences in $\R^\Z$.
It is now sufficient to show that the closed Banach subspace $c_0(\Z)$
(consisting of those sequences whose limit is zero) contains a 
closed $T$-invariant proper subspace. However, such (even symmetric, i.e. $S_\infty(\Z)$-invariant)
subspaces exist in abundance; see e.g. \cite{G1} and \cite{G2}.
\end{proof}
\end{exas}


\begin{rems}
\begin{enumerate}
\item
For the case where $X = S^1$ and $G = \Homeo(S^1)$ see Corollary \ref{S1} below.
\item
With some more work one can show that, with 
$X = S^n, \ n= 3, 4, \dots$, or $X = Q$, the Hilbert cube,
the systems $(X,\Homeo(X))$ are affinely prime.
\end{enumerate}
\end{rems}

\vspace{.5cm}

\begin{prob}\label{probs2}
Is there a non-trivial, minimal, weakly mixing, uniquely ergodic cascade $(X,T)$ which is affinely prime ?
\end{prob}

\begin{rmk}
We note that if a cascade $(X,T)$ as in Problem \ref{probs2} exists and
$\mu$ is its unique invariant measure, then the ergodic measure preserving system
$(X,\mu,T)$ has necessarily simple spectrum. 
\end{rmk}

\vspace{.5cm}

\section{The group of M\"{o}bius transformations preserving the unit disc}

Let $G$ be the group of M\"{o}bius transformations preserving the unit disk
$D = \{z \in\C : |z| < 1\}$ (see e.g. \cite[page 72]{L}):
$$
G = \{z \mapsto \frac{az + \bar{c}}{cz + \bar{a}} : \ a\bar{a} - c\bar{c} =1\}.
$$
$G$ also acts on the circle $S^1 = \{\zeta \in \C : |\zeta| =1\}$.
As was shown in \cite{F1} the system $(S^1,G)$ is the universal minimal strongly proximal $G$-system,
$\Pi_s(G)$. Another representation of this system is as the group $PSL(2,\R)$
acting on the projective line $\mathbb{P}^1$ comprising the lines through the origin in $\R^2$.

\begin{thm}\label{main}
The system $(\mathbb{P}^1, PSL(2,\R))$ is affinely prime.
Equivalently, the group $G$ of M\"{o}bius transformations preserving the unit disk
acting on the circle $S^1$ has the LSW property. 
\end{thm}

\begin{proof}
We will work with the version where $G$ is the M\"obius group
acting on $X = S^1$.

We begin by analyzing the case of complex valued functions. Let $V$ be a closed linear subspace 
of $C(S^1, \C)$ invariant under $G$ that contains a non-constant function $f$.
For all $0 \not = n \in \Z$, the convolution of $f$ with $e^{i n\theta}$ 
$$
\int e^{i n\phi} f(\theta - \phi) \,d\phi = \int e^{i n(\theta - \psi)} f(\psi)\, d\psi = e^{i n\theta} \hat{f}(n),
$$
is also contained in $V$.
Therefore, if $\hat{f}(n) \not=0$ it follows that the function $e^{i n \theta}$ is in $V$.
As $f$ is not a constant there is some $n \not=0$ for which $\hat{f}(n) \not=0$.
We fix such an $n$ and, applying the transformation $\frac{e^{i\theta} +t}{1 +te^{i\theta}}$ to
$e^{in \theta}$, we see that for all $t$, the function $\left(\frac{e^{i\theta} +t}{1 +te^{i\theta}}\right)^n$
belongs to $V$.
Upon differentiating with respect to $t$ at $t =0$, we see that
the function $n( e^{i(n-1)\theta} - e^{i (n+1)\theta})$, and hence also
the functions  $e^{i(n-1)\theta} $ and $e^{i (n+1)\theta}$, are all in $V$.

This procedure can be iterated and we conclude that $V$ contains, either
$$
\{e^{in \theta} : n \ge 0\}, \qquad  \text{or} \qquad  \{e^{-in \theta} : n \ge 0\}, 
\qquad \text{or both}.
$$
Of course in the latter case we have $V = C(S^1,\C)$.

The first alternative happens when $V$ consists of the boundary values of analytic functions 
in $D$ which are continuous on $\bar{D}$; the second, when 
$V$ consists of the boundary values of anti-analytic functions 
in $D$ which are continuous on $\bar{D}$. 

Now, for real valued functions these first two cases do not apply since a non-constant analytic 
function cannot map the boundary to the real line. 
Thus starting with a $G$-invariant closed subspace $U \subset C(S^1,\R)$
which contains a non constant function and considering its
complexification $V = \C \otimes U$, we conclude that 
$U = C(S^1,\R)$ as claimed.
\end{proof}

\vspace{.5cm}

From Corollary \ref{unique} we now get:

\begin{cor}\label{unique1}
For $G = PSL(2,\R)$, the affine system $M(\Pi_s(G)) = M(\mathbb{P}^1)$
is the only nontrivial irreducible affine $G$-system.
\end{cor}

Another immediate consequence of Theorem \ref{main} is the following:

\begin{cor}\label{S1}
The dynamical system $(S^1, \Homeo(S^1))$ is affinely prime.
\end{cor}

\vspace{.5cm}

\begin{exa}\label{exas2}
As was shown in \cite{F1}, $\Pi_s(G)$, the universal minimal strongly proximal dynamical system
for the group $G = PSL(3,\R)$ is the {\em flag manifold}:
$$
\mathcal{F} = \{(\ell, V) : \ell \subset V \subset \R^3,
\ {\text{ where $\ell$ is a line and $V$ a plane in $\R^3$}}\}.
$$
The dynamical system $(\mathcal{F}, G)$ however is not affinely prime
since it admits  (up to conjugacy) two (isomorphic) proper factors, namely the actions
of $G$ 
on the Grassman varieties $Gr(3,1)$ and $Gr(3,2)$
consisting of the lines and planes through the origin in $\R^3$, respectively 
(both are copies of the projective plane $\mathbb{P}^2$).
More generally the corresponding flag manifold is the universal minimal strongly 
proximal dynamical system for all the groups $G = PSL(d+1,\R)$, $d \ge 2$
and a similar situation occurs. See Remark \ref{fut} below.

%
\end{exa}

\begin{rmk}\label{fut}
Let $G = PSL(d +1,\R), d\ge 2$ and $X = \mathbb{P}^{d}$ be the projective space. 
With the natural action
of $G$ on $X$ the system $(X,G)$ is minimal, strongly proximal and prime. 
In fact, we can show that these actions as well are affinely prime. 
We plan to return to this in a future work.
\end{rmk}

\begin{exa}\label{exas3}
Let $X$ denote the one-sided reduced sequences on the symbols $\{a, a^{-1}, 
\allowbreak b, b^{-1}\}$,
and let $G= F_2$, the free group on the symbols $a$ and $b$, act on $X$ by
concatenation and cancelation. The dynamical system $(X,G)$ is minimal and strongly
proximal (see e.g. \cite{Gl} pages 26 and 41). 
However it is not prime and a fortiori, not affinely prime. 
To see this let $x = a^\infty = aaa\cdots$ and 
$y = a^{-\infty} = a^{-1}a^{-1}a^{-1}\cdots$
and consider the set
$$
R = \{(gx, gy), (gy,gx) : g \in G\} \cup \Del_X, 
$$
It is easy to see that this is a closed $G$-invariant equivalence relation on $X$
and consequently the induced map $\pi : X \to X/R$ yields a proper factor of $X$.
\end{exa}

\vspace{.5cm}

\section{Harmonic functions and irreducible affine dynamical systems}

Let $G$ be the group of M\"{o}bius transformations which preserve the unit disk 
$D = \{z \in \C : |z| < 1\}$, as in Theorem \ref{main}. We let $K$ denote the subgroup of
rotations in $G$. The disk $D$ can be identified with the quotient $G/K$ by the map
$g \mapsto g(0) \in D$.
$G$ is a locally compact, unimodular group with Haar measure $dg$, and we can associate to $G$
the Banach spaces $L^1(G)$ and its dual $L^\infty(G)$. With respect to the weak$^*$
topology, $B_R$, the ball of radius $R$ centered at the origin in $L^\infty(G)$,
is compact and metrizable. The group $G$ operates on $B_R$ by $f \mapsto \ ^{g'}\!f$,
where $^{g'}\!f(g) = f(gg'),\  g, g' \in G$. 

Recall that a real valued function $h$ on $D$ is {\em harmonic}
if it satisfies the mean value property:
\begin{equation*}
h(z) = \int_{S^1} h(z + r\zeta)\, d\zeta, \quad
{\text{for every sufficiently small $r$}}.
\end{equation*}

We will show that a harmonic function 
$f(z), \ z \in D, \ |f(z)| \le R$ induces an irreducible affine dynamical system
$(Q_f, G)$ with $Q_f \subset B_R$. Moreover we will see that any irreducible 
affine subsystem $Q \subset B_R$ contains a unique function 
arising from a bounded harmonic function on $D$.
For more background and details on the topic of this section see \cite{F1}.

Given $f$ bounded harmonic on $D$, define  $\tilde{f} \in L^\infty(G)$ by
$\tilde{f}(g) = f(g(0))$. That is, $\tilde{f}$ is the function on $G$ obtained by lifting
$f$ from $G/K$ to $G$.
The mean value property of harmonic functions implies that for $z' \in D$
$$
f(0) = \int_K f(kz')\, dk.
$$
Setting $z' = g'(0)$ we have
$$
f(0) = \int_Kf(kg'(0))\,dk
$$
and since for any $g \in G$, $f\circ g$ is again harmonic
$$
f(g(0)) = \int_K f(gkg'(0))\, dk, \quad {\text{or}}
$$
\begin{equation}\label{harm1}
\tilde{f}(g) = \int_K \tilde{f}(gkg')\, dk, \quad {\text{for any}} \ g, g' \in G.
\end{equation}
Now let $Q_f$ denote the closed, convex span of $\{^{g}\!\tilde{f} : g \in G\}$
in $L^\infty(G)$.
Equation (\ref{harm1}) implies that for any $F \in Q_f$
$$
\tilde{f}(g) = \int_KF(gk)\, dk.
$$
Thus $\tilde{f}$ belongs to the closed convex span of $\{^k\!F: k \in K\}$ for any $F \in Q_f$. 
This shows that $(Q_f,G)$ is an irreducible affine system.

Now let $Q \subset L^\infty(G)$ be any invariant closed convex subset such that
$(Q,G)$ is irreducible. The universal minimal strongly proximal space, $\Pi_s(G)$ is
the unit circle $S^1$ and so, by Corollary \ref{unique}, $(M(S^1),G)$ is the universal irreducible affine
system for $G$. In $M(S^1)$ there is a unique $K$-invariant measure and it follows that
in $Q$ as well, there is a unique $K$-invariant point.
As $Q$ is a space of functions on $G$, its unique $K$ fixed point is a function
$H(g)$ satisfying $H(gk) = H(g)$ for $g \in G, k \in K$.
Thus $H$ depends on $gK$ and is the pullback of a function $h$ on $D$.
For any fixed $g' \in G$ consider the function
$$
H'(g) = \int_K H(gkg')\, dk. 
$$
We have $H' \in Q$ and for $k \in K$,
$H'(gk) = H'(g)$; so $H'$ is $K$-invariant. But this function is unique; so $H'= H$. We
have $H(g) = \int_K H(gkg')\, dk$ or
\begin{equation}\label{h}
h(g) = \int_K h(gkz')\, dk
\end{equation}
for any $z' \in D$.
But, in fact, equation (\ref{h}) characterises harmonic functions. 

This discussion, combined with Theorem \ref{main} proves the following:

\begin{thm}\label{harmonic}
There is a one-to-one correspondence between bounded (non constant) harmonic functions 
$h$ on $D$ and irreducible affine subsystems $(Q,G)$ of $L^\infty(G)$.
Namely
$$
h \longleftrightarrow Q = Q_h,
$$
where $\tilde{h}$, the lift of $h$ to $G$, is the unique $K$-invariant function in $Q$.
Moreover, all the affine systems $Q_h$ are isomorphic to the universal
irreducible affine system $(M(S^1),G) = (M(\Pi_s(G)),G)$.
\end{thm}

\end{document}